\documentclass[a4paper,11pt]{article}
\usepackage[utf8]{inputenc}
\usepackage[T1]{fontenc}
\usepackage{graphicx}
\usepackage{amsmath,amsfonts,amssymb,amsthm}
\usepackage{mathtools}
\usepackage{mathrsfs}
\usepackage{bm}
\usepackage{color}
\usepackage{fullpage}
\usepackage{bbold}
\usepackage{booktabs} 
 \usepackage{natbib}
 \usepackage{float}

\definecolor{myred}{rgb}{0.77, 0.0, 0.1}
\definecolor{crimson}{rgb}{0.86, 0.08, 0.24}
\definecolor{awesome}{rgb}{1.0, 0.13, 0.32}
\definecolor{newgreen}{rgb}{0.0,0.6,0.0}
\definecolor{malachite}{rgb}{0.04, 0.85, 0.32}
\definecolor{pastelgreen}{rgb}{0.47, 0.87, 0.47}
\definecolor{myturq}{rgb}{0.1, 0.7, 0.7}

\usepackage{hyperref}
\hypersetup{
  colorlinks,
  linkcolor=blue, %
  citecolor=blue, %
  linktoc=all %
}

\usepackage{relsize}

\renewcommand{\leq}{\leqslant}
\renewcommand{\geq}{\geqslant}
\newcommand{\mleq}{\preccurlyeq}

\newcommand{\di}{\mathrm{d}}

\newcommand{\wt}{\widetilde}
\newcommand{\wh}{\widehat}
\newcommand{\ol}{\overline}
\newcommand{\pp}{\, : \;}

\newcommand{\N}{\mathbb N}

\newcommand{\R}{\mathbb R}

\renewcommand{\ln}{\log}
\usepackage{xspace}
\newcommand{\ie}{{i.e.}\@\xspace} 
\newcommand{\eg}{e.g.\@\xspace}
\newcommand{\iid}{i.i.d.\@\xspace}

\newcommand{\set}[1]{\{ #1 \}}

\newcommand{\E}{\mathbb E}
\renewcommand{\P}{\mathbb P}

\newcommand{\ind}{\bm 1}
\newcommand{\indic}[1]{\ind ( #1 )}
\newcommand{\kl}{\mathrm{KL}}
\newcommand{\kll}[2]{\kl(#1, #2)}

\newcommand{\poissondist}{\mathsf{Poisson}}

\newcommand{\F}{\mathcal{F}}

\newtheorem{theorem}{Theorem}
\newtheorem{lemma}{Lemma}

\newtheorem{fact}{Fact}

\newtheorem{remark}{Remark}

\title{Estimation of discrete distributions with high probability under $\chi^2$-divergence}

\author{Sirine Louati\footnote{CREST, ENSAE, Palaiseau, France;  \tt{sirine.louati@ensae.fr}} 
}
\date{\today}

\begin{document}
\maketitle

\begin{abstract}
 
We investigate the high-probability estimation of discrete distributions from an \iid sample under $\chi^2$-divergence loss. Although the minimax risk in expectation is well understood, its high-probability counterpart remains largely unexplored. We provide sharp upper and lower bounds for the classical Laplace estimator, showing that it achieves optimal performance among estimators that do not rely on the confidence level. We further characterize the minimax high-probability risk for any estimator and demonstrate that it can be attained through a simple smoothing strategy. Our analysis highlights an intrinsic separation between asymptotic and non-asymptotic guarantees, with the latter suffering from an unavoidable overhead.
This work sharpens existing guarantees and advances the theoretical understanding of divergence-based estimation.

\end{abstract}

\tableofcontents{}

\section{Introduction}
\label{sec:introduction}

\subsection{Problem setting}
\label{sec:problem-setting}

Learning a discrete distribution from samples is a classical task in statistics and machine learning. In this paper, we consider \(P\) an unknown distribution over a finite set \(\{1, \dots, d\}\) defined by the vector \((p_1, \dots, p_d)\), where \(p_j\) denotes the probability of the class \(j\), and  \(X_1, \dots, X_n\) be an i.i.d.\ sample  drawn from \(P\). The objective is to construct an estimator \(\widehat{P}_n = (\widehat{p}_1, \dots, \widehat{p}_d)\) of $P$ that performs well under the \emph{\(\chi^2\)-divergence} defined as 
\begin{equation}
  \label{eq:chi2-div}
  \chi^2(P, \widehat{P}_n) = \sum_{j=1}^d \frac{(p_j - \widehat{p}_j)^2}{\widehat{p}_j},
\end{equation}
meaning that it achieves small \(\chi^2\)-divergence loss with high probability over the sample.

The \(\chi^2\)-divergence \eqref{eq:chi2-div} is a natural loss function for measuring the discrepancy between probability distributions that naturally arises in goodness-of-fit testing while providing the basis of Pearson’s \(\chi^2\)-test statistic \cite{shen2018problem}, robust statistics due to its ability to amplify mismatches when the estimated probability is small~\cite{huber2011robust}, and distribution learning estimation, where it often leads to tractable expressions and clean finite-sample bounds~\cite{paninski2003estimation, devroye2001combinatorial}.
The $\chi^2$-divergence has also emerged as a central tool in several recent advances in distribution property testing, particularly within the \emph{testing-by-learning} framework initially developed by \cite{acharya2015optimal}. Notably, \cite{arora2023near} employ this framework to design near-optimal algorithms for testing structural properties of Bayesian networks, such as their maximum in-degree. In their approach, a key step consists in learning an approximation of the unknown distribution with high probability under $\chi^2$-divergence, which then enables accurate hypothesis testing. The use of $\chi^2$-divergence is particularly well-suited in this context, due to its favorable analytical properties and sensitivity to discrepancies between distributions. This perspective highlights the importance of obtaining sharp high-probability estimation guarantees under $\chi^2$-divergence, as such results directly feed into the design of efficient and theoretically grounded testing algorithms.

Besides, a key challenge in estimating the $\chi^2$-divergence lies in its extreme sensitivity to zero-probability estimates: when $\widehat{p}_j = 0$ but $p_j > 0$, the divergence becomes infinite, inducing a much harsher penalty than alternatives such as the Kullback--Leibler divergence (\eg,~\cite{liese2006asymptotically, csiszar2004information}). In fact, $\chi^2$ diverges faster when ratios go to infinity, therefore penalizing more the under-estimation of true frequencies. This sensitivity emphasizes the need for estimators that are robust to missing mass, especially in settings where rare events carry significant weight. Such robustness is critical in tasks like anomaly detection and rare event modeling (\eg,~\citep{chandola2009anomaly, blanchard2010semi}), where failing to capture the tail of the distribution can lead to brittle models. Similar concerns have emerged in recent analyses of large language models (LLMs), which must allocate non-zero probability to infrequent yet semantically important sequences to avoid incoherent or degenerate outputs (\eg,~\citep{holtzman2019curious, kandpal2023large}). As a result, controlling the $\chi^2$-divergence provides a principled way to enhance model reliability in the presence of long-tail or out-of-distribution behavior.

To study this problem, the first estimator that naturally comes to mind is the empirical distribution \(\overline{P}_n\), defined coordinate-wise by
$\overline{p}_j = \frac{N_j}{n},  j = 1, \dots, d,$
where \(N_j = \sum_{i=1}^n \mathbf{1}_{\{X_i = j\}} \) counts the occurrences of class \(j\) in the sample.
The empirical distribution coincides with the maximum likelihood estimator (MLE) over the simplex \(\Delta_d\) of probability distributions on \(\{1, \dots, d\}\). It is well-behaved in the classical asymptotic regime where \(d\) is fixed and $n$ tends to infinity since it achieves a convergence rate of \(1/\sqrt{n}\). Besides, we have that \(2n \cdot \chi^2(P, \overline{P}_n)\) converges to a \(\chi^2\) distribution with \(d-1\) degrees of freedom.
This yields the following asymptotic guarantee: for any fixed \(d\) and any \(\delta \in (0,1)\),
\begin{equation}
  \label{eq:asymptotic-mle-chi2}
  \limsup_{n \to + \infty} \mathbb{P}_P \left( \chi^2(P, \overline{P}_n) \geq \frac{2 d + 2 \log (1/\delta)}{n} \right)
  \leq \delta
  \, .
\end{equation}
However, this guarantee is purely asymptotic and non-uniform over \(P\). In modern settings where \(d\) may be large relative to \(n\), and where high-confidence bounds are required, the empirical distribution exhibits serious limitations. In particular, \(\overline{P}_n\) assigns zero probability to unseen classes, causing infinite \(\chi^2\)-divergence whenever \(p_j > 0\) and \(N_j = 0\).
Thus, the MLE may be inadequate for estimating distributions under the $\chi^2$-divergence, due to its tendency to underestimate the frequencies of certain classes. These considerations highlight the limitations of purely asymptotic guarantees, motivating the need for a quantitative, non-asymptotic analysis. Nevertheless, this bound exhibits the optimal dependence on the dimension $d$, confidence level $1-\delta$, and sample size $n$,
and may serve as a benchmark for an ideal upper bound in discrete distribution estimation.

To counteract this issue, a classical approach is Laplace smoothing estimator, which adds one pseudo-count to each class:
\begin{equation}
  \label{eq:def-laplace}
\widehat{p}_j = \frac{N_j + 1}{n + d}, \quad j=1,\dots,d
  \, .
\end{equation}
Originally proposed by Laplace \cite{laplace1825essai}, this estimator corresponds to the Bayes predictive distribution under a uniform prior on \(\Delta_d\). 
It is widely used in universal coding (see, \eg,~\cite[p.~435]{cover2006elements}) and natural language processing (\eg,~\cite[p.~46]{jurafsky2025speech}). A closely related variant (adding $1/2$ to the counts) was also proposed by Krichevsky and Trofimov \cite{krichevsky1981performance}.

\paragraph{Main questions.}

Motivated by these considerations, we study the high-probability performance of distribution estimation under \(\chi^2\)-divergence. 
Specifically, given \(n\), \(d\), and confidence level \(1-\delta\), we define \(\tau^*(n,d,\delta)\) as the minimax threshold such that
\[
\inf_{\widehat{P}_n} \sup_{P\in\Delta_d} \mathbb{P}_P\left(\chi^2(P, \widehat{P}_n) \geq \tau\right) \leq \delta.
\]
We seek to answer the following questions:
\begin{itemize}
    \item What is the best high-probability guarantee achievable by the classical Laplace estimator?
    \item What is the minimax-optimal high-probability bound achievable by any estimator?
\end{itemize}

\subsection{Related Work}
\label{sec:related}
\paragraph{About Total Variation, Hellinger and Kullback-Leibler divergences.}

Estimating discrete distributions under various divergence measures has been the subject of significant attention across statistics, learning theory, and information theory (see, e.g.,\cite{kamath2015learning,devroye2001combinatorial,diakonikolas2016learning,canonne2020short}). Classic works have focused on total variation, Hellinger, and KL divergences~\cite{tsybakov2009nonparametric, paninski2003estimation, acharya2013competitive, daskalakis2012learning, chan2013learning}. 
For the Kullback-Leibler divergence, a tighter concentration has recently been established. \cite{agrawal2022finite} proved optimal high-probability bounds, showing that the natural plug-in estimator satisfies sharp concentration inequalities under the Hellinger loss. The work of \cite{mourtada2025estimation} established optimal bounds for the KL divergence, detailed in Table~\ref{tab:chi2-bounds} below, for both confidence-dependent and independent estimators, where the confidence-dependent result provides uniform guarantees for any estimator and is achieved by a simple modification of the Laplace estimator using a confidence-dependent smoothing level.
In parallel, \cite{van2025nearly} 
study proposed another estimator that achieves a KL divergence bound of the order \(\frac{d \log\log(d) + \log(1/\delta)\log(d)}{n}\), which is almost optimal up to the $\log \log d$ factor. Moreover, for large enough \(n\) depending on the distribution, they showed that the \(\chi^2\)-divergence loss is upper bounded by \(\frac{d + \log(1/\delta)}{n}\) which does not imply a minimax upper bound. 
In addition, previous works such as ~\cite{canonne2020short}, ~\cite{kamath2015learning} and~\cite{han2015minimax} derived upper and lower bounds for TV risk. 
Besides, the work of~\cite{chhor2024generalized} prove the exact minimax rate in expectation for TV distance without any restrictions. It is worth noting that all of the aforementioned works focus on the expected risk, rather than providing high-probability guarantees.

\paragraph{The case of the $\chi^2$-divergence.}
In contrast, high-probability estimation under $\chi^2$-divergence has received comparatively less attention. 
Indeed, high-probability bounds have remained either suboptimal or computationally intensive.
A simple bound can however be obtained via Markov’s inequality. 
In fact, whenever $\E_P[\chi^2(P, \widehat P_n)] < +\infty$, one can apply Markov’s inequality to obtain
\begin{equation}
\label{eq:chi2-markov}
\chi^2(P, \widehat P_n) \leq \frac{\E_P[\chi^2(P, \widehat P_n)]}{\delta},
\end{equation}
with probability at least $1 - \delta$. 
For the Laplace estimator, it can be easily shown that
\[
  \E_P[\chi^2(P, \widehat P_n)] \leq \frac{d-1}{n+1} \leq \frac{d}{n}.
\]
We refer to Section~\ref{annexe} for a proof of this fact, and note that this result closely parallels known guarantees for the Kullback--Leibler divergence. 
Indeed, following \cite{catoni1997mixture,mourtada2022logistic}, one has
\[
  \E_P\!\left[ \mathrm{KL}(P, \widehat P_n) \right] 
  \leq \log\!\bigg( 1 + \frac{d - 1}{n + 1} \bigg) 
  \leq \frac{d}{n}.
\]
However, the dependence of the bound \( \frac{d}{n\delta} \) on \( \delta \) is suboptimal for two reasons: first, it exhibits a polynomial dependence on \( \delta^{-1} \) rather than the more favorable logarithmic dependence typically achievable in high-probability bounds; second, the factor \( d \) appears multiplicatively rather than additively, which may lead to loose guarantees in high-dimensional settings.

\paragraph{State-of-the-art bounds and open problems.}
The best-known high-probability guarantee for the estimation of discrete distributions in terms of \(\chi^2\)-divergence is provided in \cite[Proposition~4.1]{arora2023near}. Their work surveys several approaches to this problem, including empirical risk minimization, additive smoothing, and online-to-batch conversions. Their review highlights both the statistical and computational trade-offs that arise in this setting. Besides, their result offers a sharp bound for the \(\chi^2\)-divergence, ensuring that, with high probability, the estimator \(\widehat{P}_n\) satisfies
\begin{equation}
\label{eq:best known}
\mathbb{P}_P \left(\chi^2(P, \widehat{P}_n) \leq \frac{d\log(d/\delta)}{n}\right) \geq 1 - \delta,
\end{equation}
where \(d\) is the dimension of the distribution, \(n\) is the sample size, and \(\delta\) the confidence parameter. This bound is the best known in the literature for the \(\chi^2\)-divergence in terms of high-probability guarantees. In particular, a number of sample $N$ of order $\frac{d}{\epsilon}\log(\frac{d}{\delta})$ is sufficient to learn $P$ with probability $1-\delta$ under \(\chi^2\)-divergence $\epsilon$\footnote{\ie, outputting $\widehat{P}_n$ such that $\chi^2(P, \widehat{P}_n) \leq \epsilon$ with a desired probability.}.
Moreover, \cite{arora2023near} state that the analysis also provides bounds for the Laplace estimator.
Despite this significant result, the paper also points out in Section 6 that the question of obtaining optimal bounds for \(\chi^2\)-estimation remains an open problem. The authors suggest that further improvements in the dependence on \(d\) and \(\delta\) are desirable, particularly in order to achieve more refined guarantees.
Thus, while Proposition 4.1 provides a strong high-probability guarantee, it does not fully resolve the question of optimal bounds for \(\chi^2\)-divergence estimation, leaving room for future work to refine these results.

\paragraph{Our contribution.}
In this paper, we revisit the estimation of discrete distributions under $\chi^2$-divergence from a finite sample and propose a new estimator that is computationally efficient while improving the high-probability rate in~\eqref{eq:best known}.
Our main contributions are as follows:
\begin{itemize}
    \item We establish an upper bound for the Laplace-smoothed estimator with a guarantee holding uniformly over all distributions $P$.
    \item We prove a matching minimax lower bound establishing the optimality of the Laplace estimator within the class of distribution-free confidence-independent estimators and showing that a rate of order at least $\frac{d+\log^2(\frac{1}{\delta})}{n}$ is unavoidable without $\delta$-dependence.
    \item We introduce a confidence-dependent smoothing strategy yielding a refined upper bound.
    \item We complement our upper bounds with a general lower bound showing that any estimator must incur a penalty in the deviation term above the ideal asymptotic benchmark \eqref{eq:asymptotic-mle-chi2}.
\end{itemize}
Together, these results clarify the statistical price of confidence calibration for $\chi^2$-divergence estimation, and offer a practical and theoretically grounded solution via smoothed estimators.

\paragraph{Summary of known bounds.}
We summarize the guarantees discussed above in Table~\ref{tab:chi2-bounds}. All bounds hold with probability at least $1 - \delta$.
\begin{table}[h]
\centering
\caption{High-probability guarantees for estimating discrete densities under different divergences.}
\label{tab:chi2-bounds}
\begin{tabular}{ll}
\toprule
\textbf{Divergence} & \textbf{Bound} \\
\midrule
Hellinger (Empirical distribution) \cite{agrawal2022finite} & $\frac{d+\log(\frac{1}{\delta})}{n}$ \\
Kullback-Leibler (Laplace estimator) \cite{mourtada2025estimation} & $\frac{d+\log(\frac{1}{\delta})\log\log(\frac{1}{\delta})}{n}$ \\
Kullback-Leibler (Minimax risk) \cite{mourtada2025estimation} & $\frac{d+\log(d)\log(\frac{1}{\delta})}{n}$ \\
$\chi^2$-divergence \cite{arora2023near} & $\frac{ d\log(d/\delta)}{n}$ \\
$\chi^2$-divergence (Laplace estimator) \textbf{(Theorems \ref{thm:upper-laplace} and \ref{thm:lower-bound-conf-indep-tail})} & $\frac{d+\log^2(\frac{1}{\delta})}{n}$\\
$\chi^2$-divergence (Minimax risk) \textbf{(Theorems \ref{thm:upper-bound-conf-dependent} and \ref{thm:lower-bound-minimax})} & $\frac{d}{n}+\frac{\sqrt{d}\log(\frac{1}{\delta})}{n}+\frac{d\log^2(\frac{1}{\delta})}{n^2}$\\
\bottomrule
\end{tabular}
\end{table}
\subsection{Notation}
\label{sec:notation}
We fix an integer $n \geq 1$ representing the sample size and $d \geq 2$ the number of classes of the distribution. If $A$ is a finite set, we denote by $|A|$ its cardinality. 
We denote the set of probability distributions on $[d]=\{1, \ldots, d\}$ by  
$\Delta_d = \left\{ (p_1, \ldots, p_d) \in \R_+^d : \sum_{j=1}^d p_j = 1 \right\} .$
We interpret each $p_j$ as the mass assigned to class $j$. For any $j \in \{1, \dots, d\}$, the vector $\delta_j \in \Delta_d$ denotes the Dirac distribution at $j$, corresponding to the $j$-th canonical vector of $\R^d$.
Given two elements $P = (p_1, \dots, p_d)$ and $Q = (q_1, \dots, q_d)$ in $\Delta_d$, we define the \emph{$\chi^2$-divergence} between $P$ and $Q$ as $ \chi^2({P},{Q})
  = \sum_{j=1}^d \frac{(p_j-q_j)^2}{q_j}$.
For any $u, v > 0$, we introduce the function $F(u, v) := \frac{(u - v)^2}{v}$, so that $\chi^2({P},{Q})
  = \sum_{j=1}^d F (p_j, q_j)$.

Now, fix a distribution $P \in \Delta_d$. We denote by $X_1, \dots, X_n$ an \iid sample drawn from $P$. All probability and expectation operators under $P^{\otimes n}$ will be written as $\P_P$ and $\E_P$, respectively.
For each $j \in \{1, \dots, d\}$, we define the random variable $N_j$ as the number of observations in the sample that fall into class $j$, that is, $N_j := \sum_{i=1}^n \mathbf{1}_{\{X_i = j\}}$.
An estimator is a function $\Phi: \{1, \dots, d\}^n \to \Delta_d$, mapping observed samples to a probability vector. We denote the resulting estimator by $\wh P_n := \Phi(X_1, \dots, X_n)$.

Finally, for any $\lambda > 0$, we denote by $\mathsf{Poisson}(\lambda)$ the Poisson distribution with parameter $\lambda$, whose mass function is given by $\mathsf{Poisson}(\lambda)\{k\} = \frac{e^{-\lambda} \lambda^k}{k!}$ for every $k \in \N$.

\section{Main Results}
\subsection{Optimal guarantees for the Laplace estimator}  
\label{sec:laplace}
We now establish our main theoretical guarantees for high-probability estimation under $\chi^2$-divergence for the classical Laplace estimator defined by~\eqref{eq:def-laplace}. For that, we establish an upper bound in Section~\ref{upper bound laplace} and a matching lower bound in Section~\ref{lower bound laplace}.

\subsubsection{Upper bound for the Laplace estimator}
\label{upper bound laplace}
\begin{theorem}
  \label{thm:upper-laplace}
  For any $n \geq 12, d \geq 2$ and $P \in \Delta_d$, the Laplace estimator $\wh P_n$ defined by~\eqref{eq:def-laplace} achieves the following guarantee: for any $\delta \in (e^{-n/6}, e^{-2})$, letting $C_1=25900$, we have that
  \begin{equation}
    \label{eq:upper-laplace}
    \P_P \bigg( \chi^2({P},{\wh P_n})
    \geq 
    C_1 \, \frac{ d + \log^2 (1/\delta)}{n}
    \bigg)
    \leq 4 \delta
    \, .
  \end{equation}
\end{theorem}
\begin{remark}
  Note that the above bound of Theorem~\ref{thm:upper-laplace} actually
  extends to the full range $\delta \in (0,1)$. 
  First, the regime $\delta \in (e^{-2},1)$ follows directly from the case $\delta = e^{-2}$. 
  More importantly, when $\delta \in (0, e^{-n/6})$, a deterministic bound of order $n$ holds. 
  Indeed, since 
  \[
    (p_j - \widehat p_j)^2 \leq p_j^2 + \widehat p_j^2 
    \quad \text{and} \quad 
    \widehat p_j = \frac{N_j+1}{n+d} \geq \frac{1}{n+d} \geq \frac{1}{2n} 
    \quad (n \geq d),
  \]
  we obtain
  \[
    \chi^2(P,\widehat P_n) 
    = \sum_{j=1}^d \frac{(p_j - \widehat p_j)^2}{\widehat p_j} 
    \leq \sum_{j=1}^d \widehat p_j + 2n \sum_{j=1}^d p_j^2 
    \leq 1 + 2n \sum_{j=1}^d p_j = 2n+1 .
  \]
  Hence, the upper bound of Theorem~\ref{thm:upper-laplace} remains valid in this regime as well, 
  with the actual bound stopping to degrade once $\delta < e^{-n}$. 
  In particular, Theorem~\ref{thm:upper-laplace} together with the above argument shows that 
  $\chi^2(P,\widehat P_n)$ is controlled for all $\delta \in (0,1)$. 
\end{remark}

Theorem~\ref{thm:upper-laplace} improves the previously best known high-probability bound~\eqref{eq:best known} of~\cite{arora2023near}. Our bound achieves a sharper deviation term of order \(\log^2(1/\delta)/n\), and remains computationally efficient, since the estimator can be computed in linear time in $n$.

An important feature of our bound is the appearance of a deviation term of order $\log^2(1/\delta)/n$. This term reflects the presence of a heavier tail in the risk distribution, which arises because the $\chi^2$-divergence penalizes underestimations of small probabilities, even much more strongly than the Kullback--Leibler divergence. Consequently, the tail behavior is nonstandard and falls outside the scope of moment-generating function techniques, which either diverge or yield suboptimal $n$-dependence. To overcome this difficulty, our analysis relies on a direct control of raw moments through $L^p$-norm inequalities and refined probabilistic tools~\citep{latala1997estimation}. A key step in the proof consists in sharply bounding the contribution of rare events where $\widehat{p}_j$ significantly underestimates $p_j$, which constitute the main source of error under the $\chi^2$-loss.
We refer to Section~\ref{proof:thm1} for the proof of Theorem~\ref{thm:upper-laplace}.

\subsubsection{Lower bound for confidence-independent estimators}
\label{lower bound laplace}
As shown, the high-probability distribution-free guarantee of Theorem~\ref{thm:upper-laplace} for the Laplace estimator features a deviation term of order $\log^2(1/\delta)/n$, which significantly exceeds the asymptotic benchmark \eqref{eq:asymptotic-mle-chi2}.
This additional dependence is indeed fundamental as we establish in the following theorem a matching lower bound that applies to any estimator $\widehat{P}_n = \Phi(X_1, \dots, X_n)$ that does not incorporate the confidence level $1 - \delta$ as an input. That is, any confidence-independent estimator must, in the worst case, incur a deviation term of order at least $\log^2(1/\delta)/n$.
\begin{theorem}
  \label{thm:lower-bound-conf-indep-tail}
  Let $n \geq d \geq 4000$ 
  and $\kappa \geq 1$.
  Let $\Phi : [d]^n \to \Delta_d$ be an estimator such that, given $\wh P_n = \Phi (X_1, \dots, X_n)$, we have for any $P \in \Delta_d$,
  \begin{equation}
    \label{eq:assumption-in-expectation-opt}
    \P_P \Big( \chi^2({P},{\wh P_n}) \leq \frac{\kappa d}{n} \Big)
    > 0
    \, .
  \end{equation}
  Then, for any $\delta \in (e^{- n}, e^{-6.32\kappa})$, there exists a distribution $P \in \Delta_d$ such that, 
  letting $C_2= \frac{1}{10000\kappa}$,
  \begin{equation}
    \label{eq:lower-bound-conf-indep}
    \P_P \bigg( \chi^2({P},{\wh P_n}) \geq C_2\frac{d + \log^2 (1/\delta)}{n} \bigg)
    \geq \delta
    \, ,
  \end{equation}
\end{theorem}

The proof of Theorem~\ref{thm:lower-bound-conf-indep-tail}, based on Lemma~\ref{lem:lower-bound-conf-indep-tail} and presented in Section~\ref{proof: thm2}, highlights a fundamental limitation of confidence-independent estimators for $\chi^2$-divergence. 
In fact, the assumption \eqref{eq:assumption-in-expectation-opt} ensures that $\widehat{P}_n$ achieves a $\chi^2$-error of order $d/n$ with constant probability for all discrete distributions $P$, which includes estimators satisfying $\mathbb{E}_P[\chi^2(P, \widehat{P}_n)] \leq \kappa d/n$. This covers classical procedures such as the Laplace estimator. 
This family of estimators do not adapt to the target confidence level $1 - \delta$. Theorem~\ref{thm:lower-bound-conf-indep-tail} then shows that, within this class, any high-probability control must necessarily pay a price of order $\log^2(1/\delta)/n$ in the worst case. This result thus identifies an intrinsic tradeoff for confidence-independent methods.
Combined with the upper bound of Theorem~\ref{thm:upper-laplace}, the lower bound of Theorem~\ref{thm:lower-bound-conf-indep-tail} establishes that the Laplace estimator achieves the best possible high-probability performance, in a minimax sense, among a broad class of estimators.

\subsection{Minimax-optimal guarantees for confidence-dependent estimators}
\label{sec:confidence-minimax}

In what follows, we investigate the best achievable high-probability guarantees
when allowing the estimator to explicitly depend on the target confidence parameter $\delta$. For that, we establish an upper bound in Section \ref{upper bound dep} and a matching lower bound in Section \ref{lower bound dep}.

\subsubsection{Upper bound for confidence-dependent estimators}
\label{upper bound dep}
While confidence-independent estimators such as the Laplace estimator incur a penalty in the high-probability regime, with deviation term scaling as \(\log^2(1/\delta)/n\), it remains natural to ask whether this rate is intrinsic, or if it can be improved by designing estimators that adapt to the confidence parameter \(\delta\). 

\begin{theorem}
  \label{thm:upper-bound-conf-dependent}
  For any $n \geq 12, d \geq 2$ and $\delta \in (e^{-n/6}, e^{-2})$, let the estimator $\wh P_{n,\delta} = (\wh p_1, \dots, \wh p_d)$ be defined by, for $j=1, \dots, d$,
  \begin{equation}
    \label{eq:def-conf-dependent}
    \wh p_j = \frac{N_j + \lambda_\delta}{n + \lambda_\delta d}
    \qquad \text{where} \qquad
    \lambda_\delta
    = \max \bigg\{ 1, \frac{\log (1/\delta)}{\sqrt{d}} \bigg\}
    \, .
  \end{equation}
  Then, for any $P \in \Delta_d$, letting $C_3= 91190$, we have
  \begin{equation}
    \label{eq:upper-bound-conf}
    \P_P \bigg( \chi^2({P},{\wh P_{n,\delta}})
    \geq 
    C_3 \bigg( \, \frac{d}{n}+\frac{  \sqrt{d}\log (1/\delta)}{n}+\frac{ d\log^2 (1/\delta)}{n^2}\bigg)
    \bigg)
    \leq 4 \delta
    \, .
  \end{equation}
\end{theorem}

The confidence-dependent estimator~\eqref{eq:def-conf-dependent} refines the classical Laplace estimator by adapting the smoothing strength to the target confidence level~$\delta$. When $\delta \ge e^{-\sqrt{d}}$, the parameter $\lambda_\delta$ remains equal to $1$, and the estimator coincides with the standard add-one smoothing. In contrast, when $\delta < e^{-\sqrt{d}}$, the estimator applies a stronger regularization, with $\lambda_\delta = \log(1/\delta)/\sqrt{d}$, in order to stabilize the estimation of small probabilities and prevent severe underestimation events. This adaptive regularization directly translates into the deviation bound~\eqref{eq:upper-bound-conf}, whose structure is markedly different from that obtained under other divergences such as the Kullback--Leibler \cite{mourtada2025estimation}. In particular, the bound involves three distinct terms, reflecting three probabilistic regimes for~$\delta$ and the different sources of error they capture:
\[
\chi^2(P,\widehat P_{n,\delta})
\lesssim
\frac{d}{n}
\;+\;
\frac{\sqrt{d}\,\log(1/\delta)}{n}
\;+\;
\frac{d\,\log^2(1/\delta)}{n^2}.
\]

The first term, of order $d/n$, corresponds to the baseline bound in expectation and dominates when $\delta$ is relatively large. The second term, $\sqrt{d}\log(1/\delta)/n$, arises in the intermediate regime of moderate deviations, where the aggregate effect of coordinatewise fluctuations across $d$ components contributes an additional $\sqrt{d}$ factor and a linear dependence on $\log(1/\delta)$. The third term, $d\log^2(1/\delta)/n^2$, is specific to the $\chi^2$-loss and reflects the influence of rare but severe underestimation events for small probabilities. Unlike the KL divergence, which penalizes such events logarithmically through $\log(p_j/\wh p_j)$, the $\chi^2$-divergence assigns a much heavier penalty proportional to $(\wh p_j-p_j)^2/p_j$. As a result, the upper tail of the $\chi^2$-risk distribution is substantially heavier, and controlling it requires a refined analysis that goes beyond standard moment-generating-function techniques. 
The proof of Theorem~\ref{thm:upper-bound-conf-dependent}, detailed in Section~\ref{proof:thm3}, follows a structure similar to that of Theorem~\ref{thm:upper-laplace}. As before, the key technical step is then to control the contribution of underrepresented classes to the $\chi^2$-divergence, via sharp moment bounds.

\subsubsection{Lower bound for confidence-dependent estimators}
\label{lower bound dep}

While Theorem~\ref{thm:upper-bound-conf-dependent} shows that the estimator $\widehat{P}_{n,\delta}$ improves over all confidence-independent estimators, its deviation term remains above the ideal asymptotic benchmark \eqref{eq:asymptotic-mle-chi2}. Theorem~\ref{thm:lower-bound-minimax} confirms that this overhead is not an artifact of the construction or proof technique: it is statistically unavoidable, even for confidence-dependent estimators. This highlights an inherent tradeoff in high-probability estimation under $\chi^2$-divergence that persists despite tuning the estimator to the confidence level.

\begin{theorem}
  \label{thm:lower-bound-minimax}
  Let $n \geq d \geq 5000$ and $\delta \in (e^{-n}, e^{-1})$.
  For any estimator $\Phi = \Phi_\delta : [d]^n \to \Delta_d$, there exists a distribution $P \in \Delta_d$ such that, for
  $\wh P_n = \wh P_{n, \delta} = \Phi_\delta (X_1, \dots, X_n)$, we have
  \begin{equation}
    \label{eq:lower-bound-minimax}
    \P_P \bigg( \chi^2({P},{\wh P_{n}}) \geq C_4 \bigg(\frac{d + \sqrt{d} \log (1/\delta)}{ n}+ \frac{d \log^2 (1/\delta)}{ n^2}\bigg)\bigg)
    \geq \delta
    \, ,
  \end{equation}
where $C_4= \frac{1}{5000}$.
\end{theorem}

Thus, Theorems~\ref{thm:upper-bound-conf-dependent} and~\ref{thm:lower-bound-minimax} jointly yield a tight characterization, up to universal constants, of the minimax high-probability risk in discrete distribution estimation under $\chi^2$-divergence.
Notably, the lower bound~\eqref{eq:lower-bound-minimax} reveals a gap compared to the asymptotic rate~\eqref{eq:asymptotic-mle-chi2}, highlighting a fundamental separation between asymptotic and uniform non-asymptotic guarantees.
We refer to Section~\ref{proof:thm4} for a detailed description of the proof of Theorem~\ref{thm:lower-bound-minimax}.

\section{Proof techniques}
\subsection{Lower bounds}
\label{idea lower bound}
In this section, we provide the proofs of the two high-probability lower bounds in this paper, namely Theorems \ref{thm:lower-bound-conf-indep-tail} and \ref{thm:lower-bound-minimax}.
Specifically, we start by establishing these two lemmas before concluding with the proof of each specific result.
\begin{lemma}
  \label{lem:lower-minimax-tail}
  Let $n \geq d \geq 2$ and $\delta \in (e^{-n}, e^{-1})$.
  There exists a set $\F = \F_{n, d, \delta} \subset \Delta_d$ of $d$ distributions with support size at most $2$ such that for any estimator $\wh P_n = \Phi (X_1, \dots, X_n)$, there exists a distribution $P \in \F$ satisfying
  \begin{equation}
    \label{eq:lower-minimax-tail}
    \P_P \Big( \chi^2({P},{\wh P_n}) \geq 0.15\Big(\frac{\sqrt{d}\log (1/\delta)}{n} +\frac{d\log^2 (1/\delta)}{n^2}\Big)\Big)
    \geq \delta
    \, .
  \end{equation}
\end{lemma}
Compared with Lemma~1 of \cite{mourtada2025estimation}, which establishes a lower bound for the Kullback--Leibler divergence of order $\log(d)\log(1/\delta)/n$, our result in Lemma~\ref{lem:lower-minimax-tail} exhibits a two-term structure,
\[
\frac{\sqrt{d}\,\log(1/\delta)}{n}
\;+\;
\frac{d\,\log^2(1/\delta)}{n^2}.
\]
The first term mirrors the moderate-deviation regime already captured in the KL analysis, while the second term, of order $d\log^2(1/\delta)/n^2$, arises from the heavier tail behavior specific to the $\chi^2$-divergence.  
Unlike the KL divergence, whose logarithmic penalty $\log(p_j/\wh p_j)$ moderates the effect of small-probability underestimation, the quadratic scaling in the $\chi^2$-loss amplifies the contribution of coordinates with very small $p_j$.  
As a consequence, rare events where $\wh p_j$ falls far below $p_j$ dominate the tail probability, leading to an additional logarithmic factor and a stronger dependence on $d$.

\begin{proof}[Proof of Lemma~\ref{lem:lower-minimax-tail}]
  Fix $n,d,\delta$ as in Lemma~\ref{lem:lower-minimax-tail}.
  We define the class $\F = \F_{n,d,\delta}$ as
  \begin{equation*}
    \F
    = \Big\{ P^{(j)} = \delta^{1/n} \delta_1 + \big( 1 - \delta^{1/n} \big) \delta_j : 1 \leq j \leq d \Big\}
    = \big\{ \delta_1 \big\} \cup \big\{ P^{(j)} : 2 \leq j \leq d \big\}
    \, .
  \end{equation*}
  We distinguish two regimes. First, let's suppose that $\frac{\sqrt{d}\log(1/\delta)}{n} \leq 1$.
  Let $\Phi : [d]^n \to \Delta_d$ be an estimator, and let $Q = (q_1, \dots, q_d) = \Phi (1, \dots, 1)$ denote the value of this estimator when only the first class is observed.
  Let $\alpha \in \R^+$ such that $1 - q_1 = \alpha d / n$.
  
  First, assume that $\alpha \geq \log (1/\delta)/(7\sqrt{d})$.
  In this case, if $P = P^{(1)} = \delta_1$, then $\wh P_n = Q$ almost surely, hence
  \begin{equation*}
    \chi^2({P},{\wh P_n})
    = \chi^2({\delta_1},{Q})
    =  \frac{(1-q_1)^2}{q_1} + \sum_{j=2}^d {q_j}
    \geq 1 - q_1
    = \frac{\alpha d}{n}
    \, .
  \end{equation*}
  Since $\alpha \geq \log (1/\delta)/(7\sqrt{d})$, we deduce that
  \begin{equation}
    \label{eq:lower-minimax-large-alpha}
    \chi^2({P},{\wh P_n})
    \geq \frac{\sqrt{d} \log (1/\delta)}{7 n}
    \, .
  \end{equation}
 Now, assume that $\alpha < \log (1/\delta)/(7\sqrt{d})$.
  Since
  $\sum_{j=2}^d q_j = 1 - q_1 = \alpha d / n$, there exists $2 \leq j \leq d$ such that
  \begin{equation*}
    q_j
    \leq \frac{\alpha d}{n (d-1)}
    \leq \frac{2 \alpha}{n}
    < \frac{2 \log (1/\delta)}{7 n \sqrt{d}}
    \, .
  \end{equation*}
  Let $P = P^{(j)}$, so that letting $E = \{ X_1=1, \dots, X_n = 1 \}$, we have $\P_{P^{(j)}} (E) = (\delta^{1/n})^n = \delta$.
  Under $E$, one has $\wh P_n = Q$, thus denoting $\rho = 1 - \delta^{1/n}$ we have
  \begin{equation*}
    \chi^2({P},{\wh P_n})
    = \sum_{j=2}^d \frac{(p_j-q_j)^2}{q_j}
    \geq \frac{(\rho-q_j)^2}{q_j}
    \, .
  \end{equation*}
  By convexity of the exponential function, one has $1- e^{-x} \geq (1 - e^{-1}) x$ for $x \in [0,1]$; since $\frac{\log (1/\delta)}{n} \leq 1$, this implies that
  $\rho = 1 - \exp \big(- \frac{\log (1/\delta)}{n} \big) \geq (1-e^{-1}) \frac{\log (1/\delta)}{n}$.
  We therefore have
  \begin{equation*}
    \frac{\rho}{q_j}
    \geq \frac{(1-e^{-1}) \log (\delta^{-1}) / n}{2 \log (\delta^{-1})/(7 n \sqrt{d})}
    = \frac{7 (1-e^{-1}) \sqrt{d}}{2}
    \geq 2 \sqrt{d}
    \geq e
    \, .
  \end{equation*}
  Hence, we have
  \begin{align*}
    \chi^2({P},{\wh P_n})
    &\geq \frac{(\rho-q_j)^2}{q_j} \geq \frac{\rho^2}{4q_j} \geq (1-e^{-1}) \frac{\log (1/\delta)}{4n}2 \sqrt{d}=\frac{1-e^{-1}}{2}\frac{\sqrt{d}\log (1/\delta)}{n}
      \, .
  \end{align*}
  Now, let's suppose that $\frac{\sqrt{d}\log(1/\delta)}{n} \geq 1$. Let $\Phi : [d]^n \to \Delta_d$ be an estimator, and let $Q = (q_1, \dots, q_d) = \Phi (1, \dots, 1)$ denote the value of this estimator when only the first class is observed.
  Let $\alpha \in \R^+$ such that $1 - q_1 = \alpha $. Since
  $\sum_{j=2}^d q_j = 1 - q_1 = \alpha$, there exists $2 \leq j \leq d$ such that $q_j \leq \frac{\alpha }{ d-1}$.
  Let $P$ defined as before such that it implies that 
  $\rho = 1 - \exp \big(- \frac{\log (1/\delta)}{n} \big) \geq (1-e^{-1}) \frac{\log (1/\delta)}{n}$.
  Therefore, since $d \geq 2$, we have that, $\frac{d\log (1/\delta)}{n} \geq \frac{d\log (1/\delta)}{n} \geq 1$ and so that $\frac{1}{d-1} \leq \frac{\sqrt{2}\log (1/\delta)}{n}$.
  Hence, we have
  \begin{align*}
    \chi^2({P},{\wh P_n})
    &\geq \frac{(\rho-q_j)^2}{q_j} \geq \frac{((1-e^{-1}) \frac{\log (1/\delta)}{n}-\frac{1}{d-1})^2}{\frac{1}{d-1}}\\
    &\geq (d-1)(1-e^{-1}-\sqrt{2})^2 \frac{\log ^2(1/\delta)}{n^2} \geq \frac{(1-e^{-1}-\sqrt{2})^2}{2}\frac{d\log^2 (1/\delta)}{n^2}\\
    &\geq 0.62 \frac{d\log^2 (1/\delta)}{n^2}
      \, .
  \end{align*}
  Combining the results of the two regimes and using the fact that $\max(u,u^2) \geq \frac{u+u^2}{2}$, this concludes the proof of Lemma~\ref{lem:lower-minimax-tail}.
\end{proof}

\begin{lemma}
  \label{lem:lower-bound-conf-indep-tail}
  Let $n \geq d \geq 2$ and $\kappa \geq 1$, and $\wh P_n = \Phi (X_1, \dots, X_n)$ be an estimator as in Theorem~\ref{thm:lower-bound-conf-indep-tail}.
  Then, for any $\delta \in (e^{- n}, e^{-6.32\kappa})$, there exists a distribution $P \in \Delta_d$ such that \begin{equation}
    \label{eq:lower-bound-conf-indep-tail}
    \P_P \bigg( \chi^2({P},{\wh P_n}) \geq \frac{0.04}{\kappa} \frac{\log^2 (1/\delta)}{n} \bigg)
\geq \delta. \end{equation}
\end{lemma}
\begin{proof}[Proof of Lemma~\ref{lem:lower-bound-conf-indep-tail}]  
    Fix $n,d,\delta$ as in Lemma~\ref{lem:lower-bound-conf-indep-tail}.
  Let $Q = \Phi (1, \dots, 1) = (q_1, \dots, q_d) \in \Delta_d$ be the value of the estimator when only the first class is observed.
  If $P = \delta_1$, then $\wh P_n = Q$ almost surely, and thus condition~\eqref{eq:assumption-in-expectation-opt} writes
  \begin{equation*}
    \chi^2({\delta_1},{Q})
    = \frac{(1-q_1)^2}{q_1} + \sum_{j=2}^d \frac{q_j^2}{q_j}= \frac{(1-q_1)^2}{q_1} + \sum_{j=2}^d {q_j}
    \leq \frac{\kappa d}{n}
    \, .
  \end{equation*}
  This implies that $\sum_{j=2}^d q_j
    = 1 - q_1
    \leq \frac{\kappa d}{n}$ and thus there exists $j \in \set{2, \dots, d}$ such that \newline $q_j \leq (\kappa d /n)/(d-1) \leq 2 \kappa/n$.
  Now, for $\delta \in (e^{-n}, e^{-6.32\kappa})$, let the distribution $P = P_{\delta, n} = (1 - \rho) \delta_1 + \rho \delta_j$, where $\rho = 1 - \delta^{1/n}$.
  Then, the event $E = \{ X_1 = \dots = X_n = 1 \}$ is such that 
  \begin{equation*}
    \P_P (E)
    = (1-\rho)^n
    = \delta
    \, .
  \end{equation*}
  In addition, under $E$ one has $\wh P_n = Q$, thus
  \begin{equation*}
    \chi^2({P},{\wh P_n})
    = \sum_{j=2}^d \frac{(p_j-q_j)^2}{q_j}
    \geq \frac{(\rho-q_j)^2}{q_j}
    \, .
  \end{equation*}
  By convexity of the exponential function, one has $1- e^{-x} \geq (1 - e^{-1}) x$ for $x \in [0,1]$; since $\frac{\log (1/\delta)}{n} \leq 1$, this implies that
  $\rho = 1 - \exp \big(- \frac{\log (1/\delta)}{n} \big) \geq (1-e^{-1}) \frac{\log (1/\delta)}{n}$.
  Since $\delta \leq e^{-6.32\kappa}$, we therefore have
  \begin{equation*}
    \frac{(\rho-q_j)^2}{q_j} \geq \frac{\rho^2}{4q_j}
    \geq \frac{(1-e^{-1})^2 (\frac{\log (1/\delta)}{n})^2}{8 \kappa/n}
    = \frac{(1-e^{-1})^2}{8 \kappa} \frac{\log^2 (1/\delta)}{n}
    \, .
  \end{equation*}
This concludes the proof of Lemma~\ref{lem:lower-bound-conf-indep-tail}.
\end{proof}

We now establish the proofs of Theorems \ref{thm:lower-bound-conf-indep-tail} and \ref{thm:lower-bound-minimax} using above lemmas.
\begin{proof}[Proof of Theorem~\ref{thm:lower-bound-conf-indep-tail}]
\label{proof: thm2}
  First, note that $d \geq 3300 \geq 3000 \log \big( \frac{3}{1-e^{-6.32}} \big)$, so by~\cite[Proposition~1]{mourtada2025estimation}, there exists $P \in \Delta_d$ such that
  \begin{equation*}
    \P_P \bigg( \chi^2({P},{\wh P_n}) \geq \frac{d}{4600 n} \bigg) \geq \P_P \bigg( \kll{P}{\wh P_n} \geq \frac{d}{4600 n} \bigg)
    \geq 1 - 3 \exp \Big( - \frac{d}{3000} \Big)
    \geq e^{-6.32}
    \geq \delta
    \, .
  \end{equation*}
  In parallel, Lemma~\ref{lem:lower-bound-conf-indep-tail} guarantees the existence of $P \in \Delta_d$ such that
  \begin{equation*}
    \P_P \bigg( \chi^2({P},{\wh P_n}) \geq \frac{0.04}{\kappa} \frac{\log^2 (1/\delta)}{n} \bigg)
    \geq \delta
    \, .
  \end{equation*}
  Taking the better of these two bounds depending on $d$ and $\delta$, we conclude that there exists $P \in \Delta_d$ such that, with probability at least $\delta$,
  \begin{align*}
    &\chi^2({P},{\wh P_n})
    \geq \max \bigg\{ \frac{d}{4600 n},  \frac{(1-e^{-1})^2}{8 \kappa} \frac{\log^2 (1/\delta)}{n} \bigg\} \\
    &\geq \frac{99}{100} \times \frac{d}{4600 n} + \frac{1}{100} \times\frac{0.04}{ \kappa} \frac{\log^2 (1/\delta)}{n}
      \geq \frac{d + \log^2 (1/\delta)}{10000\kappa n}
      \, .
  \end{align*}
  This establishes the claimed lower bound.
\end{proof}

\begin{proof}[Proof of Theorem~\ref{thm:lower-bound-minimax}]
\label{proof:thm4}
  Since $d \geq 5000 \geq 3000 \log (\frac{3}{1-e^{-1}})$, \cite[Proposition~1]{mourtada2025estimation} ensures the existence of $P \in \Delta_d$ such that
  \begin{equation*}
    \P_P \bigg( \chi^2({P},{\wh P_n}) \geq \frac{d}{4600n} \bigg) \geq\P_P \bigg( \kll{P}{\wh P_n} \geq \frac{d}{4600n} \bigg)
    \geq 1 - 3 \exp \Big( - \frac{d}{3000} \Big)
    \geq e^{-1}
    \geq \delta
    \, .
  \end{equation*}
  In parallel, Lemma~\ref{lem:lower-minimax-tail}, provides a construction of $P \in \Delta_d$ satisfying
  \begin{equation*}
    \P_P \bigg( \chi^2({P},{\wh P_n}) \geq 0.15\bigg(\frac{\sqrt{d}\log (1/\delta)}{n}+\frac{d\log^2 (1/\delta)}{n^2}\bigg) \bigg)
    \geq \delta
    \, .
  \end{equation*}
  Combining both results and selecting the sharper bound depending on $d$ and $\delta$, we find that there exists $P \in \Delta_d$ such that, with probability at least $\delta$,
  \begin{equation*}
    \chi^2({P},{\wh P_n})
    \geq \frac{99}{100} \times \frac{d}{4600n} + \frac{1}{100} \times 0.15\bigg(\frac{\sqrt{d}\log (1/\delta)}{n}+\frac{d\log^2 (1/\delta)}{n^2}\bigg)
    \geq \frac{d + \sqrt{d} \log (1/\delta)}{5000 n}+\frac{d\log^2 (1/\delta)}{5000 n^2}
    \, .
  \end{equation*}
\end{proof}

\subsection{Upper bounds}
\label{idea upper bound}

We now outline the main proof ideas for the two high-probability upper bounds established in this paper, namely Theorems~\ref{thm:upper-laplace} and~\ref{thm:upper-bound-conf-dependent}. 
The proofs of these results follow similar lines of reasoning, outlined in what follows.

The first step consists in decomposing the $\chi^2$-divergence between the empirical and the true distributions. 
We employ two distinct decompositions, depending on the range of the confidence parameter $\delta$. 
When $\sqrt{d}\log(1/\delta)/n \leq 1$, we use a decomposition similar to that of \cite{mourtada2025estimation}, which is well suited to this regime and leads to tight control of the main terms. 
For larger values of~$\delta$, however, this approach no longer yields sharp bounds. 
In this case, we introduce an alternative decomposition, specific to our analysis, which remains accurate and tractable over this range of~$\delta$.
We start with the first decomposition, which is the following.
\begin{lemma}
  \label{lem:decomp-risk}
   Consider the distribution
  $\wh P_n = (\wh p_1, \dots, \wh p_d)
  $
  defined by
  \begin{equation}
    \label{eq:estimator-add-constant}
    \wh p_j
    = \frac{N_j + \lambda}{n + \lambda d}
    \, , \quad j = 1, \dots, d
    \, ,
  \end{equation}
  for some $\lambda \in (0, n/d]$ that may depend on $X_1, \dots, X_n$.
  Then, if $\sqrt{d} \log(1/\delta) / n \leq 1$, we have
  \begin{equation}
    \label{eq:decomp-kl-upper}
    \chi^2({P},{ \wh P_n})
    \leq 30\sum_{j=1}^d \Big(\sqrt{\ol p_j} - \sqrt{p_j}\Big)^2+\frac{100\lambda d}{3n}   
    + \left(\frac{7}{8}\right)^2\sum_{j=1}^d  \frac{2 n p_j^2}{\lambda} \bm 1 \Big( N_j \leq \frac{n p_j}{4} \Big)    
    \, ,
  \end{equation}
where $\ol p_j= N_j/n$ corresponds to the empirical frequency of class $j=1,\dots,d$.
\end{lemma}
In order to prove this result, we introduce the following lemma. 
\begin{lemma}
  \label{lem:kl-hellinger-bounded}
  
  For every $p, q \in \R^+$ such that $q \geq p/8$, one has
  \begin{equation}
    \label{eq:kl-hellinger}
    (\sqrt{p} - \sqrt{q})^2
    \leq F (p, q)
    \leq 15 (\sqrt{p} - \sqrt{q})^2
    \, .
  \end{equation}
\end{lemma}

\begin{proof}[Proof of Lemma~\ref{lem:kl-hellinger-bounded}]
\label{proof:lemma4}
We first show that 
$F (p, q)\leq 15 (\sqrt{p} - \sqrt{q})^2.$
Let \( r = \frac{q}{p} \in \left[\frac{1}{8}, 1\right] \). We have that
\[
\frac{(p - q)^2}{q \left( \sqrt{p} - \sqrt{q} \right)^2} = \frac{(1 - r)^2}{r (1 - \sqrt{r})^2}=\left( \frac{1}{\sqrt{r}} + 1 \right)^2 := f(r).
\]
Since this function is decreasing on \( [1/8, 1] \), its maximum is attained at \( r = \frac{1}{8} \), leading to $f(r) \leq \left( \sqrt{8} + 1 \right)^2 \leq 15.$
This concludes the first part of the proof. We then show that 
$(\sqrt{p} - \sqrt{q})^2 \leq F (p, q).$
We have that 
\begin{align*}
    (\sqrt{p} - \sqrt{q})^2
    &= \bigg(\frac{(\sqrt{p}-\sqrt{q})(\sqrt{p}+\sqrt{q})}{\sqrt{p}+\sqrt{q}}\bigg)^2\\
    &= \frac{(p-q)^2}{p+q+2\sqrt{pq}}\\
    &\leq \frac{(p-q)^2}{q}=F (p, q)
    \, .
    \qedhere
  \end{align*}
\end{proof}
\begin{proof}[Proof of Lemma~\ref{lem:decomp-risk}]
\label{proof:lemma3}
For $1 \leq j \leq d$, let $N_j = \sum_{i=1}^n \indic{X_i = j}$.
In addition, let $\ol P_n = (\ol p_1, \dots, \ol p_d)$ with $\ol p_j = N_j / n$ for $j= 1, \dots, d$
denote the empirical distribution. By Lemma \ref{lem:kl-hellinger-bounded}, for any $p, q \in \R^+$, if $q \geq \frac{p}{8}$, then $ F (p, q)
    \leq 15 (\sqrt{p} - \sqrt{q})^2$. And if $q \leq p/8$, then $F (p, q) \leq p^2/q$.
  Hence, for any $p, q \in \R^+$, we have
  \begin{equation}
    \label{eq:kl-hellinger-plus-res}
    F (p, q)
    \leq 15 (\sqrt{p} - \sqrt{q})^2 + \left(\frac{7}{8}\right)^2\frac{p^2}{q} \indic{q \leq p/8}
    \, .
  \end{equation}
  It follows from this inequality that
  \begin{equation}
    \label{eq:proof-decomp-kl-1}
    \chi^2({P},{\wh P_n})
    = \sum_{j=1}^d F (p_j, \wh p_j)
    \leq 15 \sum_{j=1}^d \left( \sqrt{\wh p_j} - \sqrt{p_j} \right)^2 + \left(\frac{7}{8}\right)^2\sum_{j=1}^d  \frac{p_j^2}{\wh p_j} \ind \Big( \wh p_j \leq \frac{p_j}{8} \Big)
    \, .
  \end{equation}

  We now bound the two terms of the right-hand side of~\eqref{eq:proof-decomp-kl-1}, starting with the first one.
  For that, we follow a similar approach to that of \cite[Lemma 3]{mourtada2025estimation}.
  We obtain that
  \begin{align}
    \label{eq:proof-regularized-hellinger}
    \sum_{j=1}^d \Big(\sqrt{\wh p_j} - \sqrt{p_j} \Big)^2
    &\leq \frac{2 \lambda d}{9 n} \times \sum_{j=1}^d \ol p_j + \frac{2 \lambda d}{n} + 2 \sum_{j=1}^d \Big( \sqrt{\ol p_j} - \sqrt{p_j} \Big)^2 \nonumber \\
    &= \frac{20 \lambda d}{9 n} + 2 \sum_{j=1}^d \Big( \sqrt{\ol p_j} - \sqrt{p_j} \Big)^2
    \, .
  \end{align}

  We then turn to bounding the second term in the decomposition~\eqref{eq:proof-decomp-kl-1}. Let $\lambda_\delta= \frac{\log(1/\delta)}{\sqrt{d}}$.
  Since $\sqrt{d} \log(1/\delta) / n \leq 1$, then $n + \lambda_\delta d \leq 2 n$ and thus $\wh p_j \geq \max \{ N_j, \lambda \}/(2n)$.
      This implies that
  \begin{align}
    \label{eq:proof-decomp-bound-res-term}
    \sum_{j=1}^d \frac{p_j^2}{\wh p_j}  \ind \Big( \wh p_j \leq \frac{p_j}{8} \Big)    
    &\leq \sum_{j=1}^d \frac{p_j^2}{\lambda/(2n)} \ind \bigg( \frac{\max \{ N_j, \lambda \}}{2n} \leq \frac{p_j}{8} \bigg) \nonumber \\
    &= \sum_{j \pp p_j \geq 4 \lambda/n} \frac{2 n p_j^2}{\lambda} \ind \Big( N_j \leq \frac{n p_j}{4} \Big) \nonumber \\
    &\leq \sum_{j=1}^d \frac{2 n p_j^2}{\lambda} \ind \Big( N_j \leq \frac{n p_j}{4} \Big)
      \, .
  \end{align}
     
Plugging upper bounds~\eqref{eq:proof-regularized-hellinger} and~\eqref{eq:proof-decomp-bound-res-term} into the decomposition~\eqref{eq:proof-decomp-kl-1} concludes the proof.
\end{proof}
The second decomposition is then introduced in the lemma below.
\begin{lemma}
\label{lem:decomp-risk-new}
Consider the distribution
  $\wh P_n = (\wh p_1, \dots, \wh p_d)
  $
  defined by
  \begin{equation}
    \wh p_j
    = \frac{N_j + \lambda}{n + \lambda d}
    \, , \quad j = 1, \dots, d
    \, ,
  \end{equation}
  for some $\lambda \in (0, n/d]$ that may depend on $X_1, \dots, X_n$. Then
\begin{align}
    \chi^2(P,\wh P)
    \leq 1 + \frac{8 \lambda d}{n} +\frac{\lambda d}{n} \sum_{j=1}^d \frac{2 n p_j^2}{\lambda} \ind \Big( N_j \leq \frac{n p_j}{4} \Big) \, .
\end{align}
\end{lemma}

\begin{proof}[Proof of Lemma~\ref{lem:decomp-risk-new}]
\label{proof:lemma5}
We have the following decomposition
\begin{align}
\sum_{j=1}^d \frac{p_j^2}{\wh p_j}
= \sum_{j=1}^d \frac{p_j^2}{\wh p_j} \ind \Big( N_j < \frac{n p_j}{4} \Big)
+ \sum_{j=1}^d \frac{p_j^2}{\wh p_j} \ind \Big( N_j \geq \frac{n p_j}{4} \Big).
\end{align}
For all $j$, $\wh p_j \geq \frac{\max(N_j,\lambda)}{2 \lambda d} \geq \frac{1}{2d}$ and if $N_j \geq \frac{n p_j}{4}$, then $\wh p_j \geq \frac{N_j}{2 \lambda d} \geq \frac{n p_j}{8 \lambda d}$. 
Using that $ \ind\Big( N_j \geq \frac{n p_j}{4} \Big) \leq 1$ and that $\sum_{j=1}^d p_j=1$, we obtain that
\begin{align*}
    \label{eq:}
    \sum_{j=1}^d \frac{p_j^2}{\wh p_j}\leq  \frac{\lambda d}{n}\sum_{j=1}^d \frac{2np_j^2}{\lambda} \ind \Big( N_j \leq \frac{n p_j}{4} \Big)+ \frac{8\lambda d}{n} 
      \, .
  \end{align*}
\end{proof}
Yet, for both decompositions, the point is to control the term
\[
R_\lambda = \sum_{j=1}^d \frac{2 n p_j^2}{\lambda} \ind \Big( N_j \leq \frac{n p_j}{4} \Big),
\]
which is obviously the core of the proof and corresponds to the contribution of underestimated frequencies. Indeed, this term captures the penalty incurred from severely underestimating the frequencies of certain classes, leading to a divergence that exceeds the asymptotic benchmark. This result is given by the following lemma. 
\begin{lemma}
  \label{lem:tail-residual}
  Let $P \in \Delta_d$.
  For any $\lambda \geq 1$, let 
  \begin{equation*}
    R_\lambda
    = \sum_{j=1}^d  \frac{2 n p_j^2}{\lambda} \bm 1 \Big( N_j \leq \frac{n p_j}{4} \Big)
    \, .
  \end{equation*}
  For any $\delta \in (e^{-n/6}, e^{-2})$,
  one has  
  \begin{equation}
    \label{eq:tail-residual-laplace}
    \P_P \bigg( R_1
    \geq \frac{33550 d + 33550 \log^2 (1/\delta) }{n}
    \bigg)
    \leq 2 \delta
    \, .
  \end{equation}
  Besides, for any $\delta \in (e^{-n/6}, e^{-\sqrt{d}}]$, if $\lambda = \log (1/\delta)/\sqrt{d}$, then 
  \begin{equation}
    \label{eq:tail-residual-conf}
    \P_P \bigg( R_\lambda
    \geq 
    \frac{91190\sqrt{d} \log (1/\delta)}{n}
    \bigg)
    \leq 2 \delta
    \, .
  \end{equation}
\end{lemma}
To obtain a proof of this result, we need some additional steps following the strategy of \cite[Lemma 5]{mourtada2025estimation}. First, since $R_\lambda$ is a sum of dependent random variables, we proceed with a Poisson sampling technique to handle a sum of independent random variables. This result is given by Lemma \ref{lem:tail-residual-poisson}.
\begin{lemma}
  \label{lem:tail-residual-poisson}
  Let $\wt N_1, \dots, \wt N_d$ be independent random variables, with $\wt N_j \sim \poissondist (\lambda_j/2)$ for $j = 1, \dots, d$.
  For any $\lambda \geq 1$, let 
  \begin{equation*}
    \wt R_\lambda
    = \sum_{j=1}^d  \frac{2 \lambda_j^2}{\lambda} \bm 1 \Big( \wt N_j \leq \frac{\lambda_j}{4} \Big)
    \, .
  \end{equation*}
  For any $\delta \in (0, e^{-2})$,
  one has  
  \begin{equation}
    \label{eq:tail-residual-laplace-poisson}
    \P \Big( \wt R_1
    \geq 33550d + 33550\log^2 (1/\delta) 
    \Big)
    \leq \delta
    \, .
  \end{equation}
  In addition, for any $\delta \in (0, e^{-\sqrt{d}}]$, if $\lambda = \log (1/\delta)/\sqrt{d}$, then 
  \begin{equation}
    \label{eq:tail-residual-conf-poisson}
    \P \Big( \wt R_\lambda
    \geq 
    91190\sqrt{d}
    \log (1/\delta) 
    \Big)
    \leq \delta
    \, .
  \end{equation}
\end{lemma}
With this result, we can now provide a proof of Lemma \ref{lem:tail-residual}.
\begin{proof}[Proof of Lemma~\ref{lem:tail-residual} from Lemma~\ref{lem:tail-residual-poisson}]
\label{proof:lemma6}
We make use of the \emph{Poisson sampling} method. For this, let $(X_i)_{i \geq n+1}$ be an infinite sequence of \iid random variables drawn from $P$ independent of $X_1, \dots, X_n$, and let $N \sim \poissondist(n/2)$ be an independent Poisson random variable, also independent from the sequence $(X_i)_{i \geq 1}$. For $j = 1, \dots, d$, define
\begin{equation}
  \label{eq:poisson-count}
  \wt N_j
  = \sum_{1 \leq i \leq N} \indic{X_i = j}
  \, .
\end{equation}
A standard property of Poisson processes, which can be deduced, for instance, by combining results from~\cite[Theorem~5.6, p.~100]{mitzenmacher2017probability} and~\cite[Exercise~2.1.11, p.~55]{durrett2010probability}, ensures that the random variables $\wt N_1, \dots, \wt N_d$ are mutually independent, with each $\wt N_j$ distributed as $\poissondist(n p_j / 2)$ for $j = 1, \dots, d$. Defining the event $E = \{ N \leq n \}$, where $N = \sum_{j=1}^d \wt N_j$ and applying a standard large deviation inequality for Poisson sums (see, \eg,~\cite[p.~23]{boucheron2013concentration}), one obtains a high-probability control on the event $E$
\begin{equation}
  \label{eq:proba-poisson-sample}
  \P (E)
  = 1 - \P (N > n)
  \geq 1 - e^{-n/6}
  \geq 1-\delta
  \, .
\end{equation}
In addition, under $E$ one has $\wt N_j \leq N_j$ for $j = 1, \dots, d$, hence letting $\lambda_j = n p_j$, 
\begin{equation}
  \label{eq:domination-poisson-residual}
  R_\lambda =
  \sum_{j=1}^d \frac{2 n p_j^2}{\lambda}  \bm 1 \Big( N_j \leq \frac{n p_j}{4} \Big)  
  \leq \frac{1}{n} \sum_{j=1}^d \frac{2 \lambda_j^2}{\lambda} \bm 1 \Big( \wt N_j \leq \frac{\lambda_j}{4} \Big)
  \, .
\end{equation}
We are therefore reduced to controlling the right-hand side above, where we recall that $\wt N_1, \dots, \wt N_d$ are independent random variables with $\wt N_j \sim \poissondist (\lambda_j/2)$.
This is achieved in Lemma~\ref{lem:tail-residual-poisson} below, which concludes the proof of Lemma~\ref{lem:tail-residual}.
\end{proof}

We prove Lemma \ref{lem:tail-residual-poisson} by controlling 
\(\wt R_\lambda\), a weighted sum of independent Bernoulli variables with heterogeneous success probabilities. Classical concentration inequalities such as Bennett's are inadequate here, as the tail is dominated by the largest weights rather than the variance. To address this, we control each weighted Bernoulli separately and combine the bounds via moment inequalities. In particular, we apply the sharp moment inequality of \cite{latala1997estimation}, which provides precise control for sums of independent variables. We begin with the following preliminary tail bound for the summands of \(\wt R_\lambda\).

\begin{lemma}
  \label{lem:upper-envelope}
  For any $\lambda \geq 1$, let $W^{(\lambda)}$ be a nonnegative random variable such that $W^{(\lambda)}= \frac{W}{\lambda}$ where $W$ verifies $\P\big(W\geq t^2\big)=e^{-t/14}$ for all $t\geq 0$.
  Then, for any $j=1,\dots,d$, the random variable
  \[
    V_j^{(\lambda)}=\frac{\lambda_j^2}{\lambda}\,\mathbf 1\Big(\widetilde N_j\leq \frac{\lambda_j}{4}\Big)
  \]
  is stochastically dominated by $W^{(\lambda)}$, i.e.
  \[
    \P\big(V_j^{(\lambda)}\geq w\big)\leq \P\big(W^{(\lambda)}\geq w\big)
    \quad\text{for all }w \in \R.
  \]
\end{lemma}

\begin{proof}[Proof of Lemma~\ref{lem:upper-envelope}]
\label{proof:lemma8}
First, recall $\widetilde N_j\sim\mathrm{Poisson}(\lambda_j/2)$. By a standard Poisson lower‐tail bound (see e.g.\ \cite[p.~23]{boucheron2013concentration}) we have
  \begin{equation}
    \label{eq:poisson-tail2}
    \P\Big(\widetilde N_j\leq \frac{\lambda_j}{4}\Big)\leq \exp\big(-(1-\log 2)\,\lambda_j/4\big)\leq e^{-\lambda_j/14}.
  \end{equation}

  We check the domination $\P(V_j^{(\lambda)}\geq w)\leq \P(W^{(\lambda)}\geq w)$ for any $w \in \R$.
For $w \leq 0$, both probabilities are equal to $1$.
  If $w>0$,  writing $w= t^2$,
we distinguish two cases :
\begin{itemize}
    \item If $\lambda_j\geq t$, then, since $\P(V_j^{(\lambda)}\geq w)\leq \P(\widetilde N_j\leq \lambda_j/4)\le e^{-\lambda_j/14}$ by \eqref{eq:poisson-tail2}, we bound 
  \[
    \P\big(V_j^{(\lambda)}\geq w\big)\leq e^{-\lambda_j/14}\leq e^{-t/14}
    =\P\big(W^{(\lambda)}\geq w\big).
  \]
    \item If $\lambda_j<t$, then $\lambda_j^2/\lambda<t^2/\lambda=w$, hence $V_j^{(\lambda)}\leq \lambda_j^2/\lambda<w$ and \[\P(V_j^{(\lambda)}\geq w)=0\leq \P(W^{(\lambda)}\geq w).\] \qedhere
\end{itemize}
\end{proof}

We now consider deriving a high-probability control on the sum $\sum_{j=1}^{d} W_j^{(\lambda)}$, with i.i.d.\ nonnegative heavy-tailed summands. Our main tool is the sharp moment inequality of \cite{latala1997estimation}, which provides precise control of $L^p$ norms of sums of independent variables. The version used here, stated as Lemma~\ref{lem:latala}, follows from~\cite[Corollary~1]{latala1997estimation}. We recall that for any $p \ge 1$ and real-valued $Z$, the $L^p$ norm is $\|Z\|_p := (\E[|Z|^p])^{1/p} \in \R^+ \cup \{+\infty\}$.

\begin{lemma}[\cite{latala1997estimation}, Corollary~1]
  \label{lem:latala}
  Let $Z, Z_{1}, \dots, Z_{m}$ be \iid nonnegative random variables.
  Then, for any $p \in [1, + \infty)$,
  \begin{equation*}
    \bigg\|\sum_{i=1}^{m}Z_{i} \bigg\|_{p} 
    \leq 2e^{2} \sup
    \Big\{ \frac{p}{s}\Big(\frac{m}{p}\Big)^{1/s}  \|Z\|_{s} 
    : \max \Big( 1, \frac{p}{m} \Big) \leq s \leq p
    \Big\} \, . 
  \end{equation*}
\end{lemma}
In order to apply this lemma, we introduce Lemma \ref{lem:moment-w-lambda-new} which controls the $L^p$ norm of $W^{(\lambda)}$.
\begin{lemma}
  \label{lem:moment-w-lambda-new}
  For any $p \in [1, + \infty)$, one has
  \begin{equation}
    \label{eq:moment-w-lambda-new}
    \big\| W^{(\lambda)} \big\|_p
    \leq  2270 \frac{p^2}{\lambda}
    \, .
  \end{equation}
\end{lemma}

\begin{proof}[Proof of Lemma~\ref{lem:moment-w-lambda-new}]
\label{proof:lemma10}
 We have
  \begin{align}
    \label{eq:proof-moment-w}
    \big\| W \big\|_p^p
    &= \E \big[ W^p \big]
      = \int_{0}^\infty \P \big( W^p \geq u \big) \di u \nonumber \\
    &=  \int_{0}^{\infty} \P \big( W\geq w \big) p w^{p-1} \di w  \nonumber \\
    &=  \int_{0}^{\infty} \P \Big( W \geq t^2  \Big) p   t^{2(p-1)} 2t  \di t \nonumber \\
    &= 2p \int_{0}^{\infty} e^{-t/14} t^{2p-1} \di t
    \nonumber \\
    &= 2p 14^{2p}(2p)^{2p}
      \, ,
  \end{align}
  Then, we deduce that
  \begin{align*}
    \big\| W^{(\lambda)} \big\|_p= \frac{\big\| W \big\|_p}{\lambda}
    \leq  \frac{1}{\lambda}\Big(2p 14^{2p}(2p)^{2p}\Big)^{\frac{1}{p}}
    &\leq 2270 \frac{p^2}{\lambda}
    \, .
  \end{align*}
\end{proof}
Combining the moment estimate of Lemma \ref{lem:moment-w-lambda-new} with Lemma \ref{lem:latala}, we obtain the following control on the moments of the term that dominates $\wh R_\lambda$ when $\lambda = 1$.
\begin{lemma}
  \label{lem:moment-sum-w-laplace-new}
  For any $p \geq 1$, one has
  \begin{equation}
    \label{eq:moment-sum-w-laplace-new}
    \bigg\| \sum_{j=1}^{d} W_j^{(1)} \bigg\|_p
    \leq 33550 d + 33550 p^2
    \, .
  \end{equation}
\end{lemma}

\begin{proof}[Proof of Lemma~\ref{lem:moment-sum-w-laplace-new}]
\label{proof:lemma11}
  We start with the case where $p \geq d$.
  In this case, plugging Lemma~\ref{lem:moment-w-lambda-new} into Lemma~\ref{lem:latala} and bounding $d/p \leq 1$ gives
  \begin{align*}
    \bigg\| \sum_{j=1}^{d} W_j^{(1)} \bigg\|_p
    &\leq 2 e^2 \sup \Big\{ \frac{p}{s}\Big(\frac{d}{p}\Big)^{1/s}  
    2270 s^2  \Big\} 
      : \max \Big( 1, \frac{p}{d} \Big) \leq s \leq p \Big\} \\
    &\leq 2 e^2 \sup_{1 \leq s \leq p} \Big\{ 2270ps \Big\}
      = 4540 e^2 p^2
      \, .
  \end{align*}
  We now turn to the case where $p \leq d$.
  In this case, Lemmas~\ref{lem:latala} and~\ref{lem:moment-w-lambda-new} imply that
  \begin{equation*}
    \bigg\| \sum_{j=1}^{d} W_j^{(1)} \bigg\|_p
    \leq 4540 e^2 \sup_{1 \leq s \leq p} \Big\{ p \Big( \frac{d}{p} \Big)^{1/s}s \Big\}
    \, .
  \end{equation*}
  We bound the supremum over $1 \leq s \leq p$ as follows.
  If $1 \leq s \leq 2$, then (using that $d/p \geq 1$)
  \begin{equation*}
    p \Big( \frac{d}{p} \Big)^{1/s} s
    \leq \frac{p}{2} \Big( \frac{d}{p} \Big)^{1/s} 
    \leq \frac{p}{2} \Big( \frac{d}{p} \Big)
    = \frac{d}{2}
    \, .
  \end{equation*}
  If $s \geq 2$, we consider two cases.
  If $s \leq \sqrt{d/p}$, then
  \begin{equation*}
    p \Big( \frac{d}{p} \Big)^{1/s} s
    \leq p \Big( \frac{d}{p} \Big)^{1/2}   \sqrt{\frac{d}{p}} 
    = d
    \, .
  \end{equation*}
  Finally, if $\sqrt{d/p} \leq s \leq p$, then
  \begin{equation*}
    p \Big( \frac{d}{p} \Big)^{1/s} s
    \leq p s^{1/s} p
    \leq e^{1/e} p^2
    \, .
  \end{equation*}
  Putting things together, for any value of $p$ one has
  \begin{equation*}
    \bigg\| \sum_{j=1}^{d} W_j^{(1)} \bigg\|_p
    \leq 4540 e^2 \max \big\{ d, e^{1/e} p^2\big\}
    \, ,
  \end{equation*}
  which proves~\eqref{eq:moment-sum-w-laplace-new} after bounding the maximum by a sum and simplifying constants.
\end{proof}
We next establish a similar bound on suitably large moments of the sum.
\begin{lemma}
  \label{lem:moment-sum-w-conf}
  Assume that $p \geq \sqrt{d}$
  and that $\lambda = p/\sqrt{d}$.
  Then,
  \begin{equation}
    \label{eq:moment-sum-w-conf}
    \bigg\| \sum_{j=1}^{d} W_j^{(\lambda)} \bigg\|_p
    \leq 91190 \sqrt{d}p
    \, .
  \end{equation}
\end{lemma}
\begin{proof}[Proof of Lemma~\ref{lem:moment-sum-w-conf}]
\label{proof:lemma12}
We have that 
\begin{align*}
    \bigg\| \sum_{j=1}^{d} W_j^{(\lambda)} \bigg\|_p
    &\leq 2 e^2 \sup \Big\{ \frac{p}{s}\Big(\frac{d}{p}\Big)^{1/s}  
    2270 \frac{s^2}{\lambda}  \Big\} 
      : \max \Big( 1, \frac{p}{d} \Big) \leq s \leq p \Big\} \\
    &\leq 4540 e^2 \sup_{1 \leq s \leq p} \Big\{ \sqrt{d}\Big(\frac{d}{p}\Big)^{1/s}s \Big\}
      \, .
  \end{align*}
 Recall that for any $x>0$, we have $\ln x \leq x$. Since $p^2 \geq d$, we obtain $d/p \leq p$, and therefore $\ln\!\Big(\frac{d}{p}\Big) \leq \ln p \leq p.$
Moreover, for any $p \geq 1$,
\[
\Big(\frac{d}{p}\Big)^{1/p} = \exp\!\left(\frac{\ln(d/p)}{p}\right) \leq \exp\!\left(\frac{\ln p}{p}\right) \leq e.
\]
Combining these results, we obtain that for all $s \in [1,p]$,
\[
 \sqrt{d}\,\Big(\frac{d}{p}\Big)^{1/s}s \leq e \, p \sqrt{d}.
\]
Hence,
\[
\bigg\| \sum_{j=1}^{d} W_j^{(\lambda)} \bigg\|_p \leq 4540e^3\,p\sqrt{d}.
\]
\end{proof}
Having established these results, we can now conclude the proof of Lemma~\ref{lem:tail-residual-poisson} and thus of Lemma \ref{lem:tail-residual}.

\begin{proof}[Proof of Lemma~\ref{lem:tail-residual-poisson}]
\label{proof:lemma7}
Let $\mleq$ denote stochastic domination: $X \mleq Y$ if $\P(X \ge t) \le \P(Y \ge t)$ for all $t \in \R$. By independence of $\wt N_1, \dots, \wt N_d$, Lemma~\ref{lem:upper-envelope} and \cite[Lemma~18]{mourtada2025estimation} yield
\begin{equation}
\label{eq:stoch-domin-residual-sum}
\wt R_\lambda
= 2 \sum_{j=1}^{d} \frac{\lambda_j^2}{\lambda} \bm 1\Big(\wt N_j \le \frac{\lambda_j}{4}\Big)
\mleq 2 \sum_{j=1}^{d} W_j^{(\lambda)}.
\end{equation}
It thus suffices to control the upper tail of the sum on the right-hand side of~\eqref{eq:stoch-domin-residual-sum}. In both scenarios, namely when the smoothing parameter is fixed at $\lambda = 1$ or chosen as $\lambda = \log(1/\delta)/\sqrt{d}$, this follows from the moment bounds of Lemmas~\ref{lem:moment-sum-w-laplace-new} and~\ref{lem:moment-sum-w-conf}, combined with the standard deviation-to-tail conversion: for any real $Z$ and $p \ge 1$,
\[
\P\big(|Z| \ge e \|Z\|_p\big) \le e^{-p},
\]
by Markov's inequality. Choosing $p = \log(1/\delta)$ gives the desired tail probability, and simplifying the constants yields~\eqref{eq:tail-residual-laplace-poisson} and~\eqref{eq:tail-residual-conf-poisson}.
\end{proof}

Combining all these steps, we finally prove Theorems \ref{thm:upper-laplace} and \ref{thm:upper-bound-conf-dependent}.
\begin{proof}[Proof of Theorem~\ref{thm:upper-laplace}]
\label{proof:thm1}
  We invoke the decomposition from Lemma~\ref{lem:decomp-risk} with $\lambda = 1$.
  The first term of the decomposition is controlled via the Hellinger bound of \cite[Theorem~1.5]{agrawal2022finite}
  while the second term is equal to $100 d/3n$.
  For the third term, we apply the tail bound~\eqref{eq:tail-residual-laplace} on $R_1$ from Lemma~\ref{lem:tail-residual}. 
  Putting things together and using a union bound, we obtain, for any $\delta \in (e^{-n/6}, e^{-2})$,
  \begin{equation*}
    \P_P \bigg( \chi^2({P},{\wh P_n})
    \geq 30 \times \frac{4 d + 7 \log (1/\delta)}{n} + \frac{100 d}{3n} + \frac{25700 d + 25700 \log^2 (1/\delta) }{n}
    \bigg)
    \leq 4 \delta
    \, .
  \end{equation*}
  This yields the final bound stated in Theorem~\ref{thm:upper-laplace}.
\end{proof}
\begin{proof}[Proof of Theorem~\ref{thm:upper-bound-conf-dependent}]
\label{proof:thm3}
  We distinguish two regimes.
  If $\log (1/\delta) \leq \sqrt{d}$, then $\lambda_\delta = 1$ and $\wh P_\delta$ reduces to the Laplace estimator.
  Then, Theorem~\ref{thm:upper-laplace} guarantees, with probability at least $4 \delta$,
  \begin{equation}
    \label{eq:confidence-smaller-laplace}
    \chi^2({P},{\wh P_\delta})
    \leq 25900 \frac{d + \log^2 (1/\delta) }{n}
    \leq 51800 \frac{d}{n}
    \, .
  \end{equation}
  If instead $\log (1/\delta) > \sqrt{d}$, i.e., $\delta \in (e^{-n/6}, e^{-\sqrt{d}})$, then $\lambda_\delta = \log (1/\delta)/\sqrt{d} > 1$. In this case, we distinguish two sub-cases: 
  \begin{itemize}
      \item If $\sqrt{d} \leq \log(1/\delta) \leq \frac{n}{\sqrt{d}}$, then we follow the same steps as in the proof of Theorem~\ref{thm:upper-laplace}, with two modifications: the second term in the decomposition from Lemma~\ref{lem:decomp-risk} now equals $100\lambda_\delta d/3n = 100\sqrt{d}\log (1/\delta)/3n$ and the third term is controlled using the bound~\eqref{eq:tail-residual-conf} on $R_{\lambda_\delta}$ from Lemma~\ref{lem:tail-residual}.
  Combining these terms yields, with probability at least $1-4\delta$,
  \begin{equation*}
    \chi^2({P},{\wh P_n})
    < 30 \times \frac{4 d + 7 \log (1/\delta)}{n}  + \frac{70000 \sqrt{d} \log (1/\delta)}{n}
    \, ,
  \end{equation*}
      \item If $\log(1/\delta) \geq \frac{n}{\sqrt{d}}$, We use the decomposition from Lemma~\ref{lem:decomp-risk-new} and the third term is controlled using the bound~\eqref{eq:tail-residual-conf} on $R_{\lambda_\delta}$ from Lemma~\ref{lem:tail-residual}.
  Combining these terms yields, with probability at least $1-4\delta$,
  \begin{equation*}
    \chi^2({P},{\wh P_n})
    < 1 + \frac{8 \sqrt{d}\log (1/\delta)}{n} + \frac{91190 d \log^2 (1/\delta)}{n^2}
    \, < \frac{9 \sqrt{d}\log (1/\delta)}{n} + \frac{91190 d \log^2 (1/\delta)}{n^2},
  \end{equation*} 
  \end{itemize}
Then, upper bounding by the sum in both cases completes the proof of the desired tail bound.
\end{proof}
\bigskip
\paragraph{Acknowledgments}
I would like to express my gratitude to my two PhD supervisors for their contribution to the realisation of this work. I am indebted to Jaouad Mourtada for his help in obtaining the results and to Alexandre Tsybakov for offering invaluable advice.
{  \bibliography{biblio}

\newcommand{\etalchar}[1]{$^{#1}$}
\begin{thebibliography}{vdHOvE25}

\bibitem[ABCY23]{arora2023near}
Vipul Arora, Arnab Bhattacharyya, Cl{\'e}ment~L Canonne, and Joy~Qiping Yang.
\newblock Near-optimal degree testing for bayes nets.
\newblock In {\em 2023 IEEE International Symposium on Information Theory (ISIT)}, pages 1396--1401. IEEE, 2023.

\bibitem[ADK15]{acharya2015optimal}
Jayadev Acharya, Constantinos Daskalakis, and Gautam Kamath.
\newblock Optimal testing for properties of distributions.
\newblock {\em Advances in Neural Information Processing Systems}, 28, 2015.

\bibitem[Agr22]{agrawal2022finite}
Rohit Agrawal.
\newblock Finite-sample concentration of the empirical relative entropy around its mean.
\newblock {\em Preprint arXiv:2203.00800}, 2022.

\bibitem[AJOS13]{acharya2013competitive}
Jayadev Acharya, Ashkan Jafarpour, Alon Orlitsky, and Ananda Suresh.
\newblock A competitive test for uniformity of monotone distributions.
\newblock In {\em Artificial Intelligence and Statistics}, pages 57--65. PMLR, 2013.

\bibitem[BLM13]{boucheron2013concentration}
St{\'e}phane Boucheron, G{\'a}bor Lugosi, and Pascal Massart.
\newblock {\em Concentration Inequalities: A Nonasymptotic Theory of Independence}.
\newblock Oxford University Press, 2013.

\bibitem[BLS10]{blanchard2010semi}
Gilles Blanchard, Gyemin Lee, and Clayton Scott.
\newblock Semi-supervised novelty detection.
\newblock {\em The Journal of Machine Learning Research}, 11:2973--3009, 2010.

\bibitem[Can20]{canonne2020short}
Cl{\'e}ment~L. Canonne.
\newblock A short note on learning discrete distributions.
\newblock {\em arXiv preprint arXiv:2002.11457}, 2020.

\bibitem[Cat97]{catoni1997mixture}
Olivier Catoni.
\newblock The mixture approach to universal model selection.
\newblock Technical report, {\'E}cole Normale Sup{\'e}rieure, 1997.

\bibitem[CBK09]{chandola2009anomaly}
Varun Chandola, Arindam Banerjee, and Vipin Kumar.
\newblock Anomaly detection: A survey.
\newblock {\em ACM computing surveys (CSUR)}, 41(3):1--58, 2009.

\bibitem[CDSS13]{chan2013learning}
Siu-On Chan, Ilias Diakonikolas, Xiaorui Sun, and Rocco~A Servedio.
\newblock Learning mixtures of structured distributions over discrete domains.
\newblock In {\em Proceedings of the twenty-fourth annual ACM-SIAM symposium on Discrete algorithms}, pages 1380--1394. SIAM, 2013.

\bibitem[CKT24]{chhor2024generalized}
Julien Chhor, Olga Klopp, and Alexandre Tsybakov.
\newblock Generalized multi-view model: {Adaptive} density estimation under low-rank constraints.
\newblock {\em Preprint arXiv:2404.17209}, 2024.

\bibitem[CS04]{csiszar2004information}
Imre Csisz{\'a}r and Paul~C. Shields.
\newblock Information theory and statistics: A tutorial.
\newblock {\em Foundations and Trends in Communications and Information Theory}, 1(4):417--528, 2004.

\bibitem[CT06]{cover2006elements}
Thomas~M. Cover and Joy~A. Thomas.
\newblock {\em Elements of Information Theory}.
\newblock Wiley-Interscience, 2nd edition, 2006.

\bibitem[DDS12]{daskalakis2012learning}
Constantinos Daskalakis, Ilias Diakonikolas, and Rocco~A Servedio.
\newblock Learning k-modal distributions via testing.
\newblock In {\em Proceedings of the twenty-third annual ACM-SIAM symposium on Discrete Algorithms}, pages 1371--1385. SIAM, 2012.

\bibitem[Dia16]{diakonikolas2016learning}
Ilias Diakonikolas.
\newblock Learning structured distributions.
\newblock {\em Handbook of Big Data}, 267:10--1201, 2016.

\bibitem[DL01]{devroye2001combinatorial}
Luc Devroye and G{\'a}bor Lugosi.
\newblock {\em Combinatorial Methods in Density Estimation}.
\newblock Springer Series in Statistics. Springer-Verlag New York, 1 edition, 2001.

\bibitem[Dur10]{durrett2010probability}
Rick Durrett.
\newblock {\em Probability: theory and examples}.
\newblock Cambridge University Press, 5th edition, 2010.

\bibitem[HBD{\etalchar{+}}19]{holtzman2019curious}
Ari Holtzman, Jan Buys, Li~Du, Maxwell Forbes, and Yejin Choi.
\newblock The curious case of neural text degeneration.
\newblock {\em arXiv preprint arXiv:1904.09751}, 2019.

\bibitem[HJW15]{han2015minimax}
Yanjun Han, Jiantao Jiao, and Tsachy Weissman.
\newblock Minimax estimation of discrete distributions under $\ell_{1}$ loss.
\newblock {\em IEEE Transactions on Information Theory}, 61(11):6343--6354, 2015.

\bibitem[HR11]{huber2011robust}
Peter~J Huber and Elvezio~M Ronchetti.
\newblock {\em Robust statistics}.
\newblock John Wiley \& Sons, 2011.

\bibitem[JM25]{jurafsky2025speech}
Daniel Jurafsky and James~H. Martin.
\newblock {\em Speech and Language Processing}.
\newblock Jan.~12 Draft, 3rd edition, 2025.

\bibitem[KDR{\etalchar{+}}23]{kandpal2023large}
Nikhil Kandpal, Haikang Deng, Adam Roberts, Eric Wallace, and Colin Raffel.
\newblock Large language models struggle to learn long-tail knowledge.
\newblock In {\em International Conference on Machine Learning}, pages 15696--15707. PMLR, 2023.

\bibitem[KOPS15]{kamath2015learning}
Sudeep Kamath, Alon Orlitsky, Dheeraj Pichapati, and Ananda~Theertha Suresh.
\newblock On learning distributions from their samples.
\newblock In {\em {Proceedings of the 28th Conference on Learning Theory}}, pages 1066--1100, 2015.

\bibitem[KT81]{krichevsky1981performance}
Raphail Krichevsky and Victor Trofimov.
\newblock The performance of universal encoding.
\newblock {\em IEEE Transactions on Information Theory}, 27(2):199--207, 1981.

\bibitem[Lap25]{laplace1825essai}
Pierre-Simon~de Laplace.
\newblock {\em Essai philosophique sur les probabilit{\'e}s}.
\newblock Bachelier, fifth edition, 1825.

\bibitem[Lat97]{latala1997estimation}
Rafa{\l} Lata{\l}a.
\newblock Estimation of moments of sums of independent real random variables.
\newblock {\em The Annals of Probability}, 25(3):1502--1513, 1997.

\bibitem[LMV06]{liese2006asymptotically}
Friedrich Liese, Domingo Morales, and Igor Vajda.
\newblock Asymptotically sufficient partitions and quantizations.
\newblock {\em IEEE transactions on information theory}, 52(12):5599--5606, 2006.

\bibitem[MG22]{mourtada2022logistic}
Jaouad Mourtada and St{\'e}phane Ga{\"\i}ffas.
\newblock An improper estimator with optimal excess risk in misspecified density estimation and logistic regression.
\newblock {\em Journal of Machine Learning Research}, 23(31):1--49, 2022.

\bibitem[Mou25]{mourtada2025estimation}
Jaouad Mourtada.
\newblock Estimation of discrete distributions in relative entropy, and the deviations of the missing mass.
\newblock {\em arXiv preprint arXiv:2504.21787}, 2025.

\bibitem[MU17]{mitzenmacher2017probability}
Michael Mitzenmacher and Eli Upfal.
\newblock {\em Probability and computing: {Randomization} and probabilistic techniques in algorithms and data analysis}.
\newblock Cambridge University Press, 2017.

\bibitem[Pan03]{paninski2003estimation}
Liam Paninski.
\newblock Estimation of entropy and mutual information.
\newblock {\em Neural computation}, 15(6):1191--1253, 2003.

\bibitem[SNP{\etalchar{+}}18]{shen2018problem}
W~Shen, J~Neyman, E~Pearson, G~Bolch, S~Greiner, H~de~Meer, K~Trivedi, R~Sahner, K~Trivedi, A~Puliafito, et~al.
\newblock On the problem of the most efficient tests of statistical hypotheses.
\newblock {\em Interfaces}, 48(3):285--289, 2018.

\bibitem[Tsy09]{tsybakov2009nonparametric}
Alexandre~B. Tsybakov.
\newblock {\em Introduction to nonparametric estimation}.
\newblock Springer, 2009.

\bibitem[vdHOvE25]{van2025nearly}
Dirk van~der Hoeven, Julia Olkhovskaia, and Tim van Erven.
\newblock Nearly minimax discrete distribution estimation in kullback-leibler divergence with high probability.
\newblock {\em arXiv preprint arXiv:2507.17316}, 2025.

\end{thebibliography}
  \bibliographystyle{alpha}
} 
\newpage
\appendix
\section{Additional results and proofs}
\label{annexe}
\begin{fact}
Let $P=(p_1,\dots,p_d)$ be a probability vector and let
\[
\widehat p_j=\frac{N_j+1}{n+d},\qquad j=1,\dots,d,
\]
be the Laplace (add-one) estimator built from a sample of size $n$. Then, 
\[
\E_P\big[\chi^2(P,\widehat P_n)\big]
= \frac{d-1}{n+1}-\frac{n+d}{n+1}\sum_{j=1}^d p_j(1-p_j)^{n+1},
\]
and in particular
\[
\E_P\big[\chi^2(P,\widehat P_n)\big]\leq\frac{d}{n}.
\]
\end{fact}

\begin{proof}
The $\chi^2$-divergence between $P$ and $\widehat P_n$ gives that
\[
\chi^2(P,\widehat P_n)
=\sum_{j=1}^d \frac{(p_j-\widehat p_j)^2}{\widehat p_j}
=\sum_{j=1}^d \frac{p_j^2}{\widehat p_j}-1.
\]
Then, we have that
\[
\E_P\big[\chi^2(P,\widehat P_n)\big]
= \sum_{j=1}^d p_j^2\,\E_P\!\left[\frac{1}{\widehat p_j}\right]-1
= \sum_{j=1}^d p_j^2\,\E_P\!\left[\frac{n+d}{N_j+1}\right]-1
= (n+d)\sum_{j=1}^d p_j^2\,\E_P\!\left[\frac{1}{N_j+1}\right]-1.
\]
It remains to compute $\mathbb{E}[1/(N+1)]$ for $N\sim\mathrm{Bin}(n,p)$. Using the identity $\frac{1}{k+1}=\int_0^1 t^k\,dt$, we have by Fubini-Tonelli that
\begin{align*}
\E_P\!\left[\frac{1}{N+1}\right]
&=\sum_{k=0}^n \frac{1}{k+1}\binom{n}{k}p^k(1-p)^{n-k}
= \int_0^1 \sum_{k=0}^n \binom{n}{k}(pt)^k(1-p)^{n-k}\,dt\\
&= \int_0^1\big((1-p)+pt\big)^n\,dt\\
&= \int_{1-p}^{1} s^n \frac{ds}{p}\,dt\\
&= \frac{1-(1-p)^{\,n+1}}{(n+1)p}.
\end{align*}
Plugging this back into the expectation of the $\chi^2$-divergence, we obtain that
\begin{align*}
\E_P\big[\chi^2(P,\widehat P_n)\big]
&= (n+d)\sum_{j=1}^d p_j^2\cdot\frac{1-(1-p_j)^{n+1}}{(n+1)p_j}-1\\
&= \frac{n+d}{n+1}\sum_{j=1}^d p_j\big(1-(1-p_j)^{n+1}\big)-1 \\
&= \frac{d-1}{n+1}-\frac{n+d}{n+1}\sum_{j=1}^d p_j(1-p_j)^{n+1}.
\end{align*}
Since the sum $\sum_{j=1}^d p_j(1-p_j)^{n+1}$ is nonnegative, we obtain that
\[
\E_P\big[\chi^2(P,\widehat P_n)\big]\leq \frac{d-1}{n+1}\leq\frac{d}{n},
\]
for all $n\geq1$, which yields the stated rate.
\end{proof}

\end{document}